\documentclass[12pt, reqno]{amsart}

\usepackage{amsfonts,amssymb,amscd,amsmath,enumerate,verbatim, color}
\usepackage[latin1]{inputenc}
\usepackage{amscd}
\usepackage{latexsym}
\usepackage{mathptmx}
\usepackage{multicol}

\def\NZQ{\mathbb}               

\def\RR{{\NZQ R}}

\def\frk{\mathfrak}               

\def\Phi{{\frk N}}
\def\ab{{\mathbf a}}

\def\xb{{\mathbf x}}
\def\yb{{\mathbf y}}

\def\opn#1#2{\def#1{\operatorname{#2}}} 
\opn\chara{char} 
\opn\length{\ell} 
\opn\pd{pd} 
\opn\rk{rk}
\opn\projdim{proj\,dim} 
\opn\injdim{inj\,dim} 
\opn\rank{rank}
\opn\depth{depth} 
\opn\grade{grade} 
\opn\height{height}
\opn\embdim{emb\,dim} 
\opn\codim{codim}

\opn\Tr{Tr} 
\opn\bigrank{big\,rank}
\opn\superheight{superheight}
\opn\lcm{lcm}
\opn\trdeg{tr\,deg}
\opn\reg{reg} 
\opn\lreg{lreg} 
\opn\ini{in} 
\opn\lpd{lpd}
\opn\size{size}
\opn\mult{mult}
\opn\dist{dist}
\opn\cone{cone}
\opn\lex{lex}
\opn\rev{rev}
\opn\div{div} \opn\Div{Div} \opn\cl{cl} \opn\Cl{Cl}
\opn\Spec{Spec} \opn\Supp{Supp} \opn\supp{supp} \opn\Sing{Sing}
\opn\Ass{Ass} \opn\Min{Min}
\opn\Ann{Ann} \opn\Rad{Rad} \opn\Soc{Soc}
\opn\Syz{Syz} \opn\Im{Im} \opn\Ker{Ker} \opn\Coker{Coker}
\opn\Am{Am} \opn\Hom{Hom} \opn\Tor{Tor} \opn\Ext{Ext}
\opn\End{End} \opn\Aut{Aut} \opn\id{id} \opn\ini{in}

\opn\nat{nat}
\opn\pff{pf}
\opn\Pf{Pf} \opn\GL{GL} \opn\SL{SL} \opn\mod{mod} \opn\ord{ord}
\opn\Gin{Gin}
\opn\Hilb{Hilb}\opn\adeg{adeg}\opn\std{std}\opn\ip{infpt}
\opn\Pol{Pol}
\opn\sat{sat}
\opn\Var{Var}
\opn\Gen{Gen}

\opn\aff{aff} \opn\con{conv} \opn\relint{relint} \opn\st{st}
\opn\lk{lk} \opn\cn{cn} \opn\core{core} \opn\vol{vol}
\opn\link{link} \opn\star{star}
\opn\gr{gr}

\def\pot#1#2{#1[\kern-0.28ex[#2]\kern-0.28ex]}

\opn\dirlim{\underrightarrow{\lim}}
\opn\inivlim{\underleftarrow{\lim}}
\def\Implies{\ifmmode\Longrightarrow \else
        \unskip${}\Longrightarrow{}$\ignorespaces\fi}
\def\implies{\ifmmode\Rightarrow \else
        \unskip${}\Rightarrow{}$\ignorespaces\fi}
\def\iff{\ifmmode\Longleftrightarrow \else
        \unskip${}\Longleftrightarrow{}$\ignorespaces\fi}

\let\:=\colon
\newtheorem{Theorem}{Theorem}[section]
\newtheorem{Lemma}[Theorem]{Lemma}

\newtheorem{Proposition}[Theorem]{Proposition}

\theoremstyle{definition}

\let\epsilon\varepsilon
\let\phi=\varphi
\let\kappa=\varkappa
\opn\dis{dis}
\opn\height{height}
\opn\dist{dist}
\def\pnt{{\raise0.5mm\hbox{\large\bf.}}}

\opn\Lex{Lex}

\begin{document}

\title{The number of $4$-cycles and the cyclomatic number of \\
a finite simple graph
}
\author{Takayuki Hibi, Aki Mori and Hidefumi Ohsugi}

\address{Takayuki Hibi,
Department of Pure and Applied Mathematics,
Graduate School of Information Science and Technology,
Osaka University,
Suita, Osaka 565-0871, Japan}
\email{hibi@math.sci.osaka-u.ac.jp}

\address{Aki Mori,
Department of Pure and Applied Mathematics,
Graduate School of Information Science and Technology,
Osaka University,
Suita, Osaka 565-0871, Japan} 
\email{akimk1214@gmail.com }

\address{Hidefumi Ohsugi,
Department of Mathematical Sciences,
School of Science,
Kwansei Gakuin University,
Sanda, Hyogo 669-1337, Japan} 
\email{ohsugi@kwansei.ac.jp}

\subjclass[2010]{05C30}
\keywords{finite simple graph,  number of cycles, cyclomatic number}

\begin{abstract}
Let $G$ be a finite connected simple graph with $n$ vertices and $m$ edges. 
We show that, when $G$ is not bipartite, 
the number of $4$-cycles contained in $G$ is
at most $\binom{m-n+1}{2}$.
We further provide a short combinatorial proof of the bound $\binom{m-n+2}{2}$ which 
holds for bipartite graphs.
\end{abstract}

\maketitle

\section{Introduction}
In enumerative combinatorics on finite graphs, counting the number of prescribed subgraphs contained in a finite simple graph is one of the most traditional problems. 
One of the earliest paper on this topic is due to Erd\H{o}s.
Erd\H{o}s  \cite{Erdos} determined the maximum number of complete subgraphs $K_t$ contained in $K_r$-free graphs.
There are many papers in this field of study.
Early references for general graphs are \cite{Alon01, Alon02}.
Alon \cite{Alon01, Alon02} studied the number of subgraphs of prescribed type of graphs with
$m$ edges.
As examples of the results on particular graphs, we give some references on the number of paths in graphs.
Ahlswede and Katona \cite{AK} studied the maximum number of 2-edge paths in graphs with $n$ vertices and $m$ edges. 
Bollob\'{a}s and Sarkar \cite{BS} determined the maximum number of 4-edge paths in graphs with $m$ edges. 
Nagy \cite{Nagy} studied the minimum/maximum number of 4-edge paths in graphs with given edge density.

Throughout this paper, we assume that 
every graph is finite and simple.
Let $G$ be a connected graph with $n$ vertices and $m$ edges.  
Let $c(G)$ denote the number of cycles contained in $G$.
There are a large literature 
on the number of prescribed subgraphs contained in a graph
even if we restrict a prescribed subgraph to a cycle.
For example, Ahrens \cite{Ahrens} proved that
$$
m-n+1 \le c(G) \le 2^{m-n+1} -1,
$$  
where $m-n+1$ is called the {\em cyclomatic number} or the {\em first Betti number} of $G$.
Let $c_4(G)$ denote the number of $4$-cycles in $G$.  
Alon \cite[Special case of Corollary 2.1]{AlonBip} showed that, for every fixed $\varepsilon>0$, and for any graph with $n$ vertices
and at least $\varepsilon n^2$ edges, $(1/2+o(1))\binom{n}{2}^2 (2 \varepsilon)^4 \le c_4(G)$,
where the $o(1)$ terms tend to $0$ as $n$ tends to infinity.
On the other hand, 
Hakimi and Schmeichel \cite[Theorem 2]{HS} showed that,
if $G$ is a planar graph with $n \ge 4$ vertices, then
$c_4(G) \le (n^2+3n-22)/2$.

In the present paper, we study 
upper bounds for $c_4(G)$ for a graph $G$
in terms of the cyclomatic number $m-n+1$ of $G$.  
Applying a result on commutative algebra given by Herzog et al.
\cite[Corollary~2.6]{HSZ} to the {\em edge ring} (\cite[Chapter~5]{binomialideals}) of $G$,
it follows that
\begin{equation}
\label{Sanda}
c_4(G)
 \leq 
\begin{cases}
\binom{m-n+2}{2} & \mbox{if } G  \mbox{ is bipartite},\\
\binom{m-n+1}{2} + k_4(G) & \mbox{otherwise,}
\end{cases}
\end{equation}
where $k_4(G)$ is the number of complete graphs $K_4$ in $G$.
We give the details in Appendix~\ref{omake} for readers who are interested in
the bridges between the algebraic result and the combinatorial statement.

The first main purpose of the present paper is 
to give a purely combinatorial proof of (\ref{Sanda})
for the bipartite case, and characterise the extremal graphs.

\begin{Proposition}
\label{anotherproof}
Let $G$ be a connected bipartite graph with $n$ vertices and $m$ edges.
Then 
\begin{equation}
\label{Sydney}
c_4(G) \le \binom{m-n+2}{2}.
\end{equation}
When $G$ has no vertices of degree $1$, 
equality holds if and only if $G$ is the complete bipartite graph $K_{2, n-2}$.
\end{Proposition}

The second main purpose of the present paper is to give
an improvement of (\ref{Sanda}) for the non-bipartite case.

\begin{Theorem}
\label{main}
Let $G$ be a connected graph with $n$ vertices and $m$ edges.
Suppose that $G$ has at least one odd cycle.  Then 
\begin{equation}
\label{Boston}
c_4(G) \le \binom{m-n+1}{2}.
\end{equation}
\end{Theorem}

Note that equality holds in (\ref{Boston}) if $G$ is 
the complete graph $K_n$ on $n \ge 3$ vertices.
In fact, 
\begin{equation}
c_4(K_n)  
= 3 \cdot \binom{n}{4}
=\frac{n(n-1)(n-2)(n-3)}{8}= \binom{\binom{n}{2}-n+1}{2}. \label{completegrapheformula}
\end{equation}
In addition, if $G$ is a graph 
obtained by gluing $C_4$ and an odd cycle
along a vertex, then $m-n=1$ and 
$c_4(G) = 1 = \binom{m-n+1}{2}$.

The present paper is organized as follows.
In Section 2, after we introduce some results on lattice polytopes arising from graphs,
a purely combinatorial proof of Proposition \ref{anotherproof} will be achieved.
Finally, in Section 3, we will show Theorem \ref{main} for  nonbipartite graphs.

\section{Upper bounds for bipartite graphs}

In the present section, we give a purely combinatorial proof of Proposition \ref{anotherproof} for bipartite graphs.
First, we introduce some results on lattice polytopes associated with graphs.
Given a finite set of vectors $X=\{\ab_1,\dots,\ab_m\} \subset \RR^n$,
the set
\begin{eqnarray}
{\rm conv}(X) : = \left\{ \sum_{i=1}^m \lambda_i \ab_i \in \RR^n\ \left| \ 
0 \le \lambda_i \in \RR, \  \sum_{i=1}^m \lambda_i  =1 \right. \right\} \label{convx}
\end{eqnarray}
is called the {\em convex hull} of $X$.
A subset $P \subset \RR^n$ is called a {\em polytope}
if there exists a finite set $X \subset \RR^n$ such that $P={\rm conv}(X)$.
The {\it dimension} of a polytope $ {\rm conv} (X)$ in (\ref{convx}) is
the dimension of a vector space
that is a translate of 
$
\left\{ \sum_{i=1}^m \lambda_i \ab_i \in \RR^n\ \left| \ 
\lambda_i \in \RR, \ \sum_{i=1}^m \lambda_i  =1 \right. \right\}.
$
A subset $F$ of a polytope $P \subset \RR^n$ is called a \ {\it face} of $P$
if there exists a vector ${\bf w} \in \RR^n$ such that 
$F = \{ \xb \in P \mid {\bf w} \cdot \yb \le {\bf w} \cdot \xb 
\mbox{ for any } \yb \in P\}$.
It is known that each face of a polytope is again a polytope.
In addition, there are only finitely many faces of a polytope.
The graph $G(P)$ of a polytope $P \subset \RR^n$ is a graph 
whose vertex set consists of 0-dimensional faces of $P$
and whose edge set consists of 1-dimensional faces of $P$.
If $\alpha$ is a vertex of $G(P)$ whose neighbours are
$\alpha_1,\dots,\alpha_s$ in $G(P)$, then 
it is not difficult to see that
$$
P \subset 
\left\{
\left.
\alpha +   \sum_{i=1}^s \lambda_i (\alpha_i - \alpha) \in \RR^n\ \right| \ 0\le \lambda_i \in \RR  
\right\}
.$$
From this fact, $\dim P$ is equal to
the dimension of the vector space spanned by
$\alpha_1 - \alpha, \ldots, \alpha_s - \alpha$.
Hence we have a fundamental fact on $G(P)$.

\begin{Lemma}
\label{graphGP}
Let $P \subset \RR^n$ be a $d$-dimensional polytope.
Then 
every vertex of $G(P)$ has degree at least $d$.
\end{Lemma}

A $d$-dimensional polytope $P \subset \RR^n$ is called {\em simple} if 
every vertex of the graph $G(P)$ of $P$ has degree $d$.
A $d$-dimensional polytope $P \subset \RR^n$ is called a {\em simplex} if 
$G(P)$ has $d+1$ vertices.
Since the graph $G(P)$ of a $d$-dimensional simplex $P$ is the complete graph $K_{d+1}$,
any simplex is simple.
See \cite[Chapter 3]{Zbook} for details on graphs of polytopes.

Given a graph $G$ on the vertex set $\{v_1,\ldots,v_n\}$
with the edge set $E(G)$, let ${\mathcal P}_G \subset \RR^n$ denote the convex hull of
$\{{\bf e}_i + {\bf e}_j \ | \ \{v_i, v_j\} \in E(G)\}$, where ${\bf e}_i$ is
the $i$th unit coordinate vector in $\RR^n$.
The polytope ${\mathcal P}_G$ is called the {\em edge polytope} of $G$.
See \cite[Chapter 5]{binomialideals} for details on edge polytopes.
A characterization of graphs whose edge polytope is a simplex is known  \cite[Lemma 5.5]{binomialideals}.
Simple edge polytopes are classified in Ohsugi and Hibi \cite[Corollary 5.4]{simpleedgepolytope}.

\begin{Proposition}
\label{simplepolytope}
Let $G$ be a connected graph.
Then, 
\begin{itemize}
    \item[{\rm (i)}] 
${\mathcal P}_G$ is simple if and only if 
either ${\mathcal P}_G$ is a simplex or $G$ is a complete bipartite graph{\rm ;}
    \item[{\rm (ii)}] 
${\mathcal P}_G$ is a simplex if and only if either $G$ is a tree or
$G$ contains exactly one odd cycle and it is a unique cycle of $G$. 
\end{itemize}
\end{Proposition}

Let $f_1(G)$ be the number of the $1$-dimensional faces of ${\mathcal P}_G$,
that is, the number of the edges of $G({\mathcal P}_G)$.
Several bounds for $f_1(G)$ are given in Hibi et al. \cite{HMOS}
and Tran and Ziegler \cite{TZ}.
In particular, the following proposition appears in Tran and Ziegler \cite[Proposition 9]{TZ}.

\begin{Proposition}
\label{TZ}
Let $G$ be a graph with $m$ edges.
Then $$f_1(G) = \frac{m(m-1)}{2} - 2c_4(G)+3k_4(G).$$
\end{Proposition}

Note that Propositions~\ref{simplepolytope} (i) and \ref{TZ} are proved by graph theoretical arguments based on the following
fact:
two distinct vertices 
${\bf e}_i + {\bf e}_j$,
${\bf e}_k + {\bf e}_\ell$ of 
${\mathcal P}_G$ are adjecent in $G({\mathcal P}_G)$
if and only if the induced subgraph of $G$ on the vertex
set $\{ v_i, v_j,v_k,v_\ell\}$ has no $C_4$.
By using Propositions~\ref{simplepolytope} and \ref{TZ},
we have the following proposition.

\begin{Proposition}
\label{fromTZ}
Let $G$ be a connected graph with $n$ vertices and $m$ edges.
\begin{itemize}
\item[(i)]
If $G$ is a bipartite graph, then 
$$
c_4(G) \le \frac{m(m-n+1)}{4},
$$
and equality holds if and only if $G$ is a tree or a complete bipartite graph.

\item[(ii)]
If $G$ has at least one odd cycle and no $K_4$, then 
$$
c_4(G) \le \frac{m(m-n)}{4},
$$
and equality holds if and only if $G$ contains exactly one cycle and the length of the cycle is odd.
\item[(iii)]
If $G$ has at least one $K_4$, then 
$$
c_4(G) < \frac{m(m-n)}{4} + \frac{3}{2}k_4(G).
$$

\end{itemize}
\end{Proposition}

\begin{proof}
By Proposition \ref{TZ}, we have
\begin{equation}
\label{fromTZshiki}
c_4(G) = \frac{1}{2} \left(\frac{m(m-1)}{2} - f_1(G)+3k_4(G)\right)
= \frac{m(m-1)}{4} -  \frac{1}{2}f_1(G) + \frac{3}{2}k_4(G).
\end{equation}
It is known \cite[Lemmas 5.2 and 5.4]{binomialideals} that the dimension of ${\mathcal P}_G$ is
$$
\dim {\mathcal P}_G =
\begin{cases}
n-2 & \mbox{if } G \mbox{ is bipartite,}\\
n-1 &  \mbox{otherwise,}
\end{cases}
$$
and that 
$\{{\bf e}_i + {\bf e}_j \ | \ \{v_i, v_j\} \in E(G)\}$ is the set of all vertices (0-dimensional faces) of
 ${\mathcal P}_G$.
Hence, the graph $G({\mathcal P}_G)$ of ${\mathcal P}_G$ has $m$ vertices.
From Lemma \ref{graphGP}, the degree of each vertex of $G({\mathcal P}_G)$ is at least  $\dim {\mathcal P}_G$.
%
Therefore, 
$$
f_1(G) \ge 
\begin{cases}
\frac{m(n-2)}{2} &\mbox{if }  G \mbox{ is bipartite,}\\
\frac{m(n-1)}{2} &  \mbox{otherwise.}\\
\end{cases}
$$
Thus by (\ref{fromTZshiki}), we have
$$
c_4(G)  \le 
\begin{cases}
\frac{m(m-n+1)}{4} & \mbox{if } G \mbox{ is bipartite,}\\
\frac{m(m-n)}{4} + \frac{3}{2}k_4(G)&  \mbox{otherwise,}\\
\end{cases}
$$
and equality holds if and only if the edge polytope of $G$ is a simple polytope,
which is characterized in Proposition \ref{simplepolytope}.
In particular, if the edge polytope of $G$ is simple,
then $G$ has no $K_4$.
\end{proof}

Now we turn to proving Proposition \ref{anotherproof}, which we do inductively. First we set some notation.
A vertex $v$ of a connected graph $G$ is called a {\em cut vertex}
if the graph obtained by the removal of $v$ from $G$ is disconnected.
An edge $e$ of a connected graph $G$ is called a {\em bridge}
if the graph obtained by the removal of $e$ from $G$ is disconnected.
Let
$$
\varepsilon(G) =
\begin{cases}
1 & \mbox{if } G \mbox{ is bipartite},\\
0 & \mbox{otherwise}.
\end{cases}
$$
By the following lemma, we may assume that a graph $G$ 
has no cut vertices.

\begin{Lemma}
\label{2connected}
Suppose that connected graphs $G_1$ and $G_2$ have exactly one common vertex,
and each $G_i$ has $n_i$ vertices and $m_i$ edges.
Let $G $
be the graph $G_1 \cup G_2$ with $n=n_1 +n_2-1$ vertices and $m=m_1+m_2$ edges.
If
$$
c_4(G_i) \le 
\binom{m_i-n_i+1+\varepsilon(G_i)}{2}
$$
for $i=1,2$,
then we have 
\begin{equation}
\label{awase}
c_4(G) \le \binom{m -n +1+\varepsilon(G)}{2},
\end{equation}
and equality holds if and only if 
$c_4(G_i) =
\binom{m_i-n_i+1 + \varepsilon(G_i)}{2}$ for $i=1,2$ and 
either {\rm (i)} at least one $G_i$ is a tree or 
{\rm (ii)} $G_1$ contains exactly one cycle and the length of the cycle is odd
 and $G_2$ is bipartite, or vice versa.
\end{Lemma}

\begin{proof}
Note that 
$\varepsilon(G) = \varepsilon(G_1) \cdot \varepsilon(G_2)$.
Since every 4-cycle is included in a block of $G$, we have
$ c_4(G) =c_4(G_1) + c_4(G_2) $.

\bigskip

\noindent
{\bf Case 1} ($G_2$ is a tree){\bf .}
Then $c_4(G_2) =
\binom{m_2-n_2+1 + \varepsilon(G_2)}{2}=0$,
$c_4(G) =c_4(G_1)$ and $\varepsilon(G) = \varepsilon(G_1) $.
Hence (\ref{awase}) holds.
Moreover, equality holds if and only if $c_4(G_1) = \binom{m_1-n_1+1 + \varepsilon(G_1)}{2}$.

\bigskip

\noindent
{\bf Case 2} ($G_i$ is not a tree for $i \in \{1,2\}$){\bf .}
It then follows that 
\begin{eqnarray*}
 & & \binom{(m_1+m_2) - (n_1+n_2 -1) +1+\varepsilon(G)}{2} - c_4(G) \\
&\ge&
\binom{(m_1+m_2) - (n_1+n_2) +2+\varepsilon(G)}{2}\\
&& \ \ \ \ \ \  -\binom{m_1-n_1+1+\varepsilon(G_1)}{2}-\binom{m_2-n_2+1+\varepsilon(G_2)}{2}\\
\\
&=& 
\begin{cases}
(m_1-n_1+1) (m_2-n_2+1) & \mbox{ if } \varepsilon(G_1)=\varepsilon(G_2),\\
(m_1-n_1) (m_2-n_2+1) & \mbox{ if } \varepsilon(G_1) =0 \mbox{ and } \varepsilon(G_2)=1,\\
(m_1-n_1+1) (m_2-n_2) & \mbox{ if } \varepsilon(G_1) =1 \mbox{ and } \varepsilon(G_2)=0
\end{cases}
\end{eqnarray*}
is nonnegative.
Thus we have equation (\ref{awase}).

Since $G_i$ is not a tree for $i=1,2$, 
both $m_1-n_1+1$ and $m_2-n_2+1$ are positive.
If $\varepsilon(G_1)=\varepsilon(G_2)$, then
$(m_1-n_1+1) (m_2-n_2+1) $ is not zero.
Suppose that $\varepsilon(G_1) =0$ and $\varepsilon(G_2)=1$.
Then $(m_1-n_1) (m_2-n_2+1)=0$
 if and only if 
$m_1=n_1$
 if and only if 
$G_1$ contains exactly one cycle and the length of the cycle is odd
 and $G_2$ is bipartite.
The case when $\varepsilon(G_1) =1$ and $\varepsilon(G_2)=0$ is similar.
\end{proof}

We now give a purely combinatorial proof of 
Proposition~\ref{anotherproof} for bipartite graphs.

\begin{proof}[Proof of Proposition~\ref{anotherproof}]
We proceed by induction on $m$.
Let $G$ be a bipartite graph such that
$
c_4(G) > \binom{m-n+2}{2}
$
with minimal $m$.
By Lemma \ref{2connected}, we may assume that 
$G$ has no cut vertices (and so no bridges).
Let $G'$ be a (connected) subgraph of $G$ obtained by deleting an edge $e_0$ of $G$.
Let $c_4(e_0)$ denote the number of 4-cycles of $G$ containing $e_0$.
By the induction hypothesis, we have
$$
c_4(G') \le \binom{(m-1)-n+2}{2}
= 
\binom{m-n+2}{2} - (m-n+1).
$$
Hence $m-n+2 \le c_4(e_0)$.
Since all edges satisfy this condition,
$$
c_4(G) \ge \frac{m(m-n+2)}{4}
>\frac{m(m-n+1)}{4}.
$$
This contradicts Proposition \ref{fromTZ}
and hence (\ref{Sydney}) holds.

On the other hand, we have
$$
c_4(K_{2,n-2})  = \binom{n-2}{2}
=
\binom{2(n-2) - n + 2}{2} 
.$$
Conversely, 
let $G$ be a bipartite graph with
$
c_4(G) =  \binom{m-n+2}{2}
$
having no vertices of degree $1$.
By Lemma \ref{2connected}, $G$ has no cut vertices (and so no bridges)
since $G$ has no vertices of degree $1$ and no odd cycles.
Let $G'$ be a (connected) subgraph of $G$ obtained by deleting an edge $e_0$ of $G$.
Since 
$$
c_4(G') \le \binom{(m-1)-n+2}{2}
= c_4(G) - (m-n+1)
$$
holds, it follows that $m-n+1 \le c_4(e_0)$ for each edge $e_0 \in E(G)$.
Thus $\frac{m(m-n+1)}{4} \le c_4(G)$.
By Proposition \ref{fromTZ} (i), 
$c_4(G) =\frac{m(m-n+1)}{4}$ and hence $G$ is a complete bipartite graph.
Since $G$ is not a tree,  $m-n+1\neq 0$.
Hence
$\frac{(m-n+1)(m-n+2)}{2}=\frac{m(m-n+1)}{4}$ if and only if $m=2(n-2)$.
It then follows that $G$ is the complete bipartite graph $K_{2, n-2}$.
\end{proof}

\section{Upper bound for nonbipartite graphs}

In this section,
we prove the main theorem  (Theorem \ref{main}) of the present paper.
Theorem~\ref{main} will be proved by induction on the number of edges,
and Propositions \ref{deltafour} and \ref{induction} will play 
important roles in the proof.
In order to show these propositions, 
we use the following theorem, lemma, and propositions
for nonbipartite graphs.
\begin{itemize}
    \item 
    Proposition \ref{withoutk4} states that 
    $c_4(G) \le \binom{m-n+1}{2}$
    if $k_4(G) \le 1$.
    The proof of this proposition is similar to the proof of Proposition \ref{anotherproof}.
    In addition, the argument in the proof will be used for the proof of Proposition \ref{induction}.

    \item
    Lemma \ref{large:w(G)} states that 
    $c_4(G) \le \binom{m-n+1}{2}$
    if $K_{n-1}$ is a subgraph of $G$.

\item
Theorem \ref{msthm} is Motzkin--Straus Theorem for the maximum value 
of the function
$\sum_{\{v_i,v_j\}\in E(G)} x_i x_j$.
We will give a sketch of the proof for the readers.

\item
Proposition \ref{omegadelta} gives 
an upper bound for $c_4(G)$
in terms of $m$, $n$, the minimum degree,
the clique number, and the sum $\Sigma_2$
of the squares of the degrees for $G$.
We will give a proof of
Proposition \ref{omegadelta}
by using
Motzkin--Straus Theorem.
This will be useful because upper bounds
(\ref{referee2}) and (\ref{referee1}) for $\Sigma_2$
in terms of 
$m$, $n$, the minimum degree,
and the maximum degree are known.

\item
Proposition \ref{deltafour}
states that $c_4(G) \le \binom{m-n+1}{2}$
if 
the minimum degree,
and the maximum degree of a graph 
satisfy some conditions.
The main tool for the proof
is Proposition \ref{omegadelta}.

\item
Proposition \ref{induction}
states that $c_4(G) \le \binom{m-n+1}{2}$
if $G$ satisfies some conditions
on the minimum degree,
the maximum degree,
and any nonbipartite subgraph $G \setminus v$
or $G\setminus e$ with 
$v \in V(G)$, $e \in E(G)$ satisfies such an inequality.
Again, the main tool for the proof
is Proposition \ref{omegadelta}.
    
\end{itemize}

First, we show that, if $G$ has at most one $K_4$, then $c_4(G) \le \binom{m-n+1}{2}$.
Given a graph $G$, a {\em block} of $G$ is a maximal connected subgraph of $G$ with no cut vertices.

\begin{Proposition}
\label{withoutk4}
Let $G$ be a connected graph having at least one odd cycle.
If $k_4(G) \le 1$, then we have
$$
c_4(G) \le \binom{m-n+1}{2}.
$$
When $G$ has no vertices of degree $1$ and $k_4(G)=0$,
equality holds if and only if $G$ is one of the following{\rm :}
\begin{itemize}
\item[(a)]
 an odd cycle{\rm ;}
\item[(b)]
the union of $K_{2, n'}$ and a path $P$
where common vertices of $K_{2, n'}$ and $P$ are
end vertices of $P${\rm ;}

\item[(c)]
a graph whose set of blocks consists of one $K_{2,n'}$, one odd cycle and some bridges.
\end{itemize}
\end{Proposition}

\begin{proof}
We proceed by induction on $m$.
Let $G$ be a connected graph having at least one odd cycle
such that $k_4(G) \le 1$ and 
$
c_4(G) > \binom{m-n+1}{2}
$
with minimal $m$.
By Lemma \ref{2connected}, we may assume that $G$ has no cut vertices.
Let $G'$ be a (connected) subgraph of $G$ obtained by deleting an edge $e_0$ of $G$.
Then we claim the following:

\medskip

\noindent
{\bf Claim 1.} 
$G'$ is not bipartite.

If $G'$ is bipartite, then $e_0$ joins two vertices in the same part in $G'$
since $G$ is not bipartite.
Hence there exists no $C_4$ that contains $e_0$.
Thus
$$
c_4(G) = c_4(G') \le \binom{(m-1)-n+2}{2}
= \binom{m-n+1}{2}
$$
which is a contradiction.
Therefore $G'$ is not bipartite.

\medskip

\noindent
{\bf Claim 2.} 
$m-n+1 \le c_4(e_0)$.

By the induction hypothesis, we have
$$
c_4(G') \le \binom{(m-1)-n+1}{2}
= \binom{m-n+1}{2} - (m-n).
$$
Hence $m-n+1 \le c_4(e_0)$.

\medskip

Since every edge $e_0$ of $G$ satisfies $m-n+1 \le c_4(e_0)$,
$$
c_4(G) \ge \frac{m(m-n+1)}{4}.
$$
However, by Proposition \ref{fromTZ},
$$
c_4(G) \le \frac{m(m-n)}{4} <  \frac{m(m-n+1)}{4},
$$
if $G$ has no $K_4$, and
$$
c_4(G) < \frac{m(m-n)}{4} + \frac{3}{2}
\le \frac{m(m-n)}{4} + \frac{m}{4}
=
\frac{m(m-n+1)}{4}
$$
if $G$ has one $K_4$ (and hence $m\ge 6$).
This is a contradiction and hence 
$c_4(G) \le \binom{m-n+1}{2}$.

On the other hand, let $G$ be a connected graph having at least one odd cycle 
such that
$
c_4(G) = \binom{m-n+1}{2}
$,
$k_4(G) = 0$,
and $G$ has no vertices of degree $1$.

\medskip

\noindent
{\bf Case 1.}
($G$ has a cut vertex.)
Since $G$ is not bipartite, $\varepsilon(G) =0$.
Since $G$ has no vertices of degree $1$,
by Lemma \ref{2connected},
$G = G_1 \cup G_2$
where 
\begin{itemize}
\item
$G_1$ and $G_2$ have exactly one common vertex,
\item
$G_1$ contains exactly one cycle and the length of the cycle is odd, and
\item
$G_2$ is a bipartite graph with $n'$ vertices and $m'$ edges such that $c_4(G_2) = \binom{m'-n'+1}{2}$.
\end{itemize}
Moreover, by Proposition~\ref{anotherproof}, 
$G_2$ is a graph whose set of blocks
consists of one $K_{2,\ell}$ and some bridges.
Thus $G$ satisfies (c).

\medskip

\noindent
{\bf Case 2.}
($G$ has  no cut vertices.)
Suppose that there exists an edge $e_0$ of $G$
such that the subgraph $G'$ of $G$ obtained by deleting an edge $e_0$ of $G$ is bipartite.
Then there exists no $C_4$ that contains $e_0$.
Thus we have
$$
c_4(G') = c_4(G) = \binom{m-n+1}{2}=
\binom{(m-1)-n+2}{2}.
$$
Since $G'$ is a bipartite graph having $n$ vertices and $m-1$ edges, by Proposition~\ref{anotherproof},
$G'$ is either a tree or the complete bipartite graph $K_{2, n'}$ 
(by removing vertices of degree $1$).
Therefore $G$ satisfies either (b) or (c).
Suppose that for any edge $e_0$ of $G$,
the subgraph $G'$ of $G$ obtained by deleting $e_0$ has at least one odd cycle.
Since $k_4(G')=0$, we have
$$
c_4(G') \le \binom{(m-1)-n+1}{2}
= c_4(G) - (m-n).
$$
Hence $m-n \le c_4(G)- c_4(G')= c_4(e_0)$ for each edge $e_0$ of $G$.
Thus $c_4(G) \ge \frac{m(m-n)}{4}$.
Since $k_4(G) = 0$, 
by Proposition \ref{fromTZ} (ii), $c_4(G) \le \frac{m(m-n)}{4}$ and hence
$c_4(G)=\frac{m(m-n)}{4}$.
From Proposition~\ref{fromTZ} (ii),
$G$ contains exactly one cycle and the length of the cycle is odd.
Since $G$ has no vertices of degree $1$, $G$ is an odd cycle.
This contradicts the hypothesis that $G'$ has at least one odd cycle.
\end{proof}

A {\em clique} of a graph $G$ is a subgraph of $G$ that is a complete graph.
The {\em clique number} $\omega(G)$ of a graph $G$ is the number of vertices in a maximum clique of $G$.
If $G$ has at least one odd cycle and $\omega(G)$ is at most $3$, then $G$ has no $K_4$
and hence $c_4(G) \le \binom{m-n+1}{2}$.
On the other hand, if $\omega(G)$ is very large, then $c_4(G) \le \binom{m-n+1}{2}$
also holds for $G$.

\begin{Lemma}
\label{large:w(G)}
Let $G$ be a connected graph with $n \ge 4$ vertices and $\omega(G) \ge n-1$.
Then $c_4(G) \le \binom{m-n+1}{2}$.
\end{Lemma}

\begin{proof}
If $\omega(G) =n$, then $G=K_n$.
This case is proved in (\ref{completegrapheformula}) of Section 1.
If $\omega(G) =n-1$, then $K_{n-1}$ is a subgraph of $G$.
Suppose that $m=\binom{n-1}{2}+k$ where $1 \le k \le n-2$.
Then
\begin{eqnarray*}
&&
\binom{m-n+1}{2} - c_4(G) \\
&=&
\frac{(\binom{n-1}{2}+k-n)(\binom{n-1}{2}+k-n+1)}{2}
-
3 \binom{n-1}{4} - (n-3) \binom{k}{2}\\
&=&
\frac{(n-4) (k-1) (n-1-k)}{2}  \ge 0.
\end{eqnarray*}
Thus $c_4(G) \le \binom{m-n+1}{2}$.
\end{proof}

Next, we give an upper bound for $c_4(G)$
in terms of several parameters on $G$ 
(Proposition \ref{omegadelta}).
The following theorem will play an important role 
for the proof of Proposition~\ref{omegadelta}. 

\begin{Theorem}[Motzkin--Straus \cite{MSthm}]
\label{msthm}
Let $G$ be a graph on the vertex set $\{v_1,\ldots,v_n\}$.
Then
$$
\max \left\{ \sum_{\{v_i,v_j\}\in E(G)} x_i x_j \ \left| \ x_i \ge 0, \sum_{i=1}^n x_i=1 \right. \right\} = \frac{1}{2} \left( 1 - \frac{1}{\omega(G)} \right)
.$$
\end{Theorem}

\begin{proof}[Sketch of Proof.]
%
Suppose that $(x_1,\ldots, x_n)$ with $x_i \ge 0$ and $\sum_{i=1}^n x_i=1$ attains the maximum for the function $\sum_{\{v_i,v_j\}\in E(G)} x_i x_j $.
If $x_i, x_j > 0$ ($i \ne j$) and $\{v_i, v_j\}$ is not an edge of $G$, then, by replacing $(x_i, x_j)$ with
either $(x_i+x_j,0)$ or $(0,x_i+x_j)$, 
we obtain a new $(x_1,\ldots, x_n)$ such that 
the value of the function does not decrease.
Thus we may assume that $S = \{ v_i \mid x_i >0\}$ is 
the vertex set of a clique of $G$.
Then
$$
\sum_{\{v_i,v_j\}\in E(G)} x_i x_j
=
\frac{1}{2}
\left(
\left(\sum_{v_i \in S} x_i\right)^2 - 
\sum_{v_i \in S} x_i^2 
\right)
=
\frac{1}{2}
\left(
1 -  \sum_{v_i \in S} x_i^2 
\right)
\le
\frac{1}{2}
\left(
1 -  \frac{1}{|S|}
\right),
$$
and the equality holds when $x_i = 1/|S|$ for all $v_i \in S$.
\end{proof}

Given a graph $G$ on the vertex set $\{v_1,\ldots,v_n\}$, 
let $\delta(G)$ denote the {\em minimum degree} of $G$,
and let $\Delta(G)$ denote the {\em maximum degree} of $G$.
That is,
$$\delta(G) = \min \{ \deg_G(v_i) \ | \ 1 \le i \le n\}, \ \ 
\Delta(G) = \max \{ \deg_G(v_i) \ | \ 1 \le i \le n \}.
$$
Let $$\Sigma_2 = \sum_{i=1}^n \deg_G(v_i)^2
\ \left(=\sum_{\{v_i, v_j\}  \in E(G)}(\deg_G(v_i) +\deg_G(v_j) )\right).$$
Several upper bounds for $\Sigma_2$ are known.
For example, if $\delta \le \delta(G) \le  \Delta(G) \le \Delta$, then
\begin{eqnarray}
\Sigma_2 
&\le& 
m \left(\frac{2m}{n-1}+ n-2 \right), \label{referee2}
\\
\Sigma_2 
&\le & 2m (\Delta + \delta ) - n \Delta \delta . \label{referee1}
\end{eqnarray}
Here (\ref{referee2}) is
given by de Caen \cite[Theorem 1]{s2s2}, and (\ref{referee1}) 
is 
%
given
by Das \cite[Theorem 4.3]{s2}:
\begin{eqnarray*}
\Sigma_2 
&=&
\sum_{i=1}^n (\deg_G(v_i) (\deg_G(v_i) -\delta) +\deg_G(v_i) \delta)\\
&\le&
\sum_{i=1}^n (\Delta (\deg_G(v_i) -\delta) +\deg_G(v_i) \delta)
=2m (\Delta + \delta ) - n \Delta \delta.    
\end{eqnarray*}


\begin{Proposition}
\label{omegadelta}
Let $G$ be a connected graph with $\delta(G)\ge 2$.
Then, for any $2 \le \alpha \le \delta(G)$ and
$\beta  \ge \omega(G)$, we have
\begin{equation}
c_4(G) \le \frac{(2m- \alpha n)^2}{8}\left(1 - \frac{1}{\beta}\right)+\frac{m(n-\alpha^2+1)}{4}
+\frac{\alpha -2}{4}
\Sigma_2. \label{referee06}
\end{equation}
In addition,
\begin{equation}
c_4(G) =  \frac{(2m-  \delta(G) n)^2}{8}\left(1 - \frac{1}{\omega(G)}\right)+\frac{m(n- \delta(G)^2+1)}{4}
+\frac{ \delta(G) -2}{4} \Sigma_2 \label{referee05}
\end{equation}
holds if and only if  $G$ is either $K_n$ or $K_{\ell, \ell}$ with $n=2\ell$.
\end{Proposition}

\begin{proof}
We will prove (\ref{referee06}) by applying 
Motzkin--Straus Theorem.
First we have to rewrite the expression
so that we can apply Motzkin--Straus Theorem.

Let $c_t (i,j)$ denote the number of $t$-cycles of $G$ containing $\{v_i, v_j\} \in E(G)$.
Then 
\begin{eqnarray*}
c_4(i,j) &\le& (\deg_G(v_i) - 1)(\deg_G(v_j) - 1)  - c_3(i,j)\\
&\le& (\deg_G(v_i) - 1)(\deg_G(v_j) - 1)  - (\deg_G(v_i)+\deg_G(v_j)-n)\\
&=& (\deg_G(v_i) - \alpha )(\deg_G(v_j) - \alpha ) 
+ (\alpha -2)(\deg_G(v_i) +\deg_G(v_j) )  +n-\alpha^2+1.
\end{eqnarray*}
Since $ \sum_{\{v_i, v_j\}\in E(G)} c_4 (i,j)= 4 c_4(G) $
and
$
\sum_{\{v_i, v_j\}  \in E(G)}(\deg_G(v_i) +\deg_G(v_j) )
=\Sigma_2
$, we have
\begin{eqnarray*}
c_4(G) &\le &
\frac{1}{4}
\sum_{\{v_i, v_j\} \in E(G)}
(\deg_G(v_i) - \alpha )(\deg_G(v_j) -\alpha )
+ \frac{m(n-\alpha^2+1)}{4} +  \frac{\alpha -2}{4} \Sigma_2 .
\end{eqnarray*}
Since $\alpha n \le \delta(G) n \le 2m$, we have $2 m- \alpha  n \ge 0$.
Equality holds if and only if $\alpha = \delta(G)$
and $G$ is a regular graph.
In this case, 
$\sum_{\{v_i, v_j\}  \in E(G)}
(\deg_G(v_i) - \alpha )(\deg_G(v_j) -\alpha )
=0$.
If $2 m- \alpha  n > 0$, then
let $$x_i = \frac{\deg_G(v_i) - \alpha}{2 m- \alpha  n}\ge 0$$
for each vertex $v_i$.
Since
$\sum_{i=1}^n (\deg_G(v_i) - \alpha )=
2 m-  \alpha n,$
we have $ \sum_{i=1}^n x_i = 1$.
By Motzkin--Straus Theorem,
$$
\sum_{\{v_i, v_j\}  \in E(G)} x_i x_j \ 
\le \ 
\frac{1}{2}\left(1-\frac{1}{\omega(G)}\right).$$
It follows that
\begin{eqnarray*}
c_4(G) &\le& 
\frac{(2m- \alpha n)^2}{8}\left(1 - \frac{1}{\omega(G)}\right)+\frac{m(n-\alpha^2+1)}{4}
+\frac{\alpha-2}{4} \Sigma_2\\
&\le &\frac{(2m- \alpha n)^2}{8}\left(1 - \frac{1}{\beta}\right)+\frac{m(n-\alpha^2+1)}{4}
+\frac{\alpha-2}{4}
\Sigma_2
\end{eqnarray*}
for any $\beta \ge \omega(G)$.

Next we will classify graphs $G$
satisfying (\ref{referee05}).
The complete graph $K_n$ has $m = n(n-1)/2$ edges and
satisfies
$\delta(K_n) =n-1$ and $\Sigma_2=n(n-1)^2$.
By substituting these values into the right-hand side of (\ref{referee05}), we have
\begin{eqnarray*}
&&\frac{(2m-  \delta(K_n) n)^2}{8}\left(1 - \frac{1}{\omega(K_n)}\right)+\frac{m(n- \delta(K_n)^2+1)}{4}
+\frac{ \delta(K_n) -2}{4}
\Sigma_2\\
&=&n (n- 1)(n- 2)(n- 3)/8\\
&=&c_4(K_n).
\end{eqnarray*}
On the other hand, if $G=K_{\ell, \ell}$, then 
$(n,m)=(2\ell, \ell^2)$, $\delta(G) =\ell$, and $\Sigma_2=2 \ell^3$.
By substituting these values into the right-hand side of (\ref{referee05}), we have
\begin{eqnarray*}
&&
\frac{(2m-  \delta(K_{\ell, \ell}) n)^2}{8}\left(1 - \frac{1}{\omega(K_{\ell, \ell})}\right)+\frac{m(n- \delta(K_{\ell, \ell})^2+1)}{4}
+\frac{ \delta(K_{\ell, \ell}) -2}{4}
\Sigma_2\\
&=&\frac{\ell^2(\ell-1)^2}{4}
=\binom{\ell}{2}^2
=c_4(K_{\ell, \ell}).
\end{eqnarray*}
Conversely, suppose that $G \notin \{ K_n, \; K_{\ell,\ell} \}$ satisfies (\ref{referee05}).
It then follows that
\begin{eqnarray}
c_4(i,j) &=&(\deg_G(v_i) - 1)(\deg_G(v_j) - 1)  - c_3(i,j) \label{cij4}\\
c_3(i,j) &=& \deg_G(v_i)+\deg_G(v_j)-n  \label{cij3}
\end{eqnarray}
for any edge $\{v_i, v_j\}$ of $G$.
From (\ref{cij3}), we have
$$
(\deg_G(v_i)-1)+(\deg_G(v_j)-1) -c_3(i,j)=n-2.
$$
Hence any vertex of $G$ is adjacent to either $v_i$ or $v_j$ (or both).
Moreover, from (\ref{cij4}), 
if $\{v_i, v_k\}$ and $\{v_j, v_\ell\}$ are edges of $G$ with $k \neq j$, $\ell \neq i$, $k \neq \ell$,
then $\{v_k,v_\ell\}$ is an edge of $G$.
Let $K_r$ ($r < n$) be a maximum clique of $G$.

\bigskip

\noindent
{\bf Case 1} ($r\ge3$).
Since $G$ is connected,
there exists an edge $\{v_i, v_j\}$ of $G$ such that $v_i$ is a vertex of $K_r$
and $v_j$ is not a vertex of $K_r$.
Then the above claim guarantees that 
$v_j$ is adjacent to all vertices of $K_r$.
This contradicts the hypothesis that $K_r$ is a maximum clique.

\bigskip

\noindent
{\bf Case 2} ($r=2$).
Let $\{v_i, v_j\}$ be an edge of $G$.
Since $G$ has no triangles,
any vertex of $G$ is adjacent to exactly one of $v_i$ and $v_j$.
Moreover, since $$c_4(i,j) = (\deg_G(v_i) - 1)(\deg_G(v_j) - 1),$$
$G$ is a complete bipartite graph $K_{\ell_1 ,\ell_2}$,
where $\ell_1 = \deg_G(v_i)$, $\ell_2=\deg_G(v_j)$.
We may assume that $\ell_2 \le \ell_1$.
Then $K_{\ell_1 ,\ell_2}$ has $\ell_1 + \ell_2$
vertices, $\ell_1 \ell_2$ edges,
and satisfies
\begin{eqnarray*}
\delta(K_{\ell_1 ,\ell_2})&=& \ell_2,\\
\omega(K_{\ell_1 ,\ell_2})&=&2,\\
\Sigma_2 &=& \ell_1 \ell_2^2 + \ell_2 \ell_1^2 =
\ell_1 \ell_2 (\ell_1 + \ell_2),\\
c_4(K_{\ell_1 ,\ell_2}) &=& \binom{\ell_1}{2} \binom{\ell_2}{2}.
\end{eqnarray*}
Hence the right-hand side of (\ref{referee05})
is
\begin{eqnarray*}
&&
\frac{(2 \ell_1 \ell_2- \ell_2 (\ell_1+\ell_2))^2}{16} 
+\frac{ \ell_1 \ell_2 \left(\ell_1+\ell_2-\ell_2^2+1\right)}{4}
+\frac{(\ell_2-2) \ell_1 \ell_2 ( \ell_1 + \ell_2 )}{4} 
\\
&=&
\frac{\ell_2^2 (\ell_1 -\ell_2)^2}{16} 
+ c_4(K_{\ell_1,\ell_2}) .
\end{eqnarray*}
Thus equality (\ref{referee05}) holds if and only if $\ell_1 = \ell_2$.
\end{proof}

As an application, we have the following propositions that 
will play an important role for a proof of the main theorem.

\begin{Proposition}
\label{deltafour}
Let $G$ be a connected graph with $n$ vertices having at least one odd cycle.
If $G$ satisfies at least one of the following conditions,
then 
$
c_4(G) \le \binom{m-n+1}{2}.
$
\begin{enumerate}
\item[{\rm (i)}]
$\delta(G) \ge 4$ and $ \Delta(G) \le \frac{3n+1}{4}${\rm ;}

\item[{\rm (ii)}]
$\delta(G) \ge 12${\rm ;}

\item[{\rm (iii)}]
$5 \le \delta(G) \le \Delta(G) \le n-2$ and $n \le 27$.
\end{enumerate}
\end{Proposition}

\begin{proof}
By Lemma \ref{large:w(G)},
we may assume that $\omega(G) \le n-2$.
If $n \le 5$, then $G$ has no $K_4$
and hence the inequality holds
from Proposition \ref{withoutk4}.
We may assume that $n \ge 6$.
Let $\delta = \delta(G)$ and $\Delta = \Delta(G)$.
We will use Proposition~\ref{omegadelta} 
together with (\ref{referee2}) and (\ref{referee1})
for the proof.

(i)
By (\ref{referee1}),
we have
$
\Sigma_2
\le 2m(\Delta+4) -4n \Delta$
since $\delta \ge 4$.
By this and Proposition \ref{omegadelta}, it follows that
\[
c_4(G) \le \frac{(m- 2 n)^2}{2}\left(1 - \frac{1}{\omega(G)}\right)+\frac{m(n-15)}{4}
+ m(\Delta+4) - 2 n \Delta.
\]
Hence
\begin{eqnarray*}
&&\binom{m-n+1}{2} - c_4(G) \\
& \ge & \binom{m-n+1}{2} - 
\left(
\frac{(m- 2 n)^2}{2}\left(1 - \frac{1}{\omega(G)}\right)+\frac{m(n-15)}{4}
+ m(\Delta+4) - 2 n \Delta
\right)\\
&=& \frac{(m-2 n) ( (m- 2 n)-\frac{\omega(G)}{2}( 4 \Delta- 3n-1) )}{2 \omega(G)}.
\end{eqnarray*}
Since $\delta \ge 4$, we have $m \ge 2n$.
Thus $\binom{m-n+1}{2} - c_4(G) \ge 0$ if
$ \Delta \le \frac{3n+1}{4}$ holds.

(ii)
Let $\delta(G) \ge 12$.
Then $n \ge 13$ and $m \ge 6n$.
By substituting $\alpha=4$ and $\beta = n-2$ in the inequality 
in Proposition \ref{omegadelta} and (\ref{referee2}),
we have
\begin{eqnarray*}
& & 
\binom{m-n+1}{2}-c_4(G)\\
& \ge &
\binom{m-n+1}{2}-
\left(
\frac{( m- 2n)^2}{2}\left(1 - \frac{1}{n-2}\right)+\frac{m(n-15)}{4}
+\frac{1}{2}
m \left(\frac{2 m}{n-1}+n-2\right)
\right)\\
&=&
\frac{1}{2(n-2)}
\left(
-\frac{ n-3 }{ n-1 }m^2 
+
\frac{ (n-3)(n+14)}{2}m
-
n (3 n^2-9 n-2)
\right).
\end{eqnarray*}
Let
$$\varphi(x) = -\frac{n-3 }{ n-1 }x^2 
+
\frac{ (n-3)(n+14)}{2}x
-
n (3 n^2-9 n-2).$$
Since $n \ge 13$,
\begin{eqnarray*}
\varphi(6n)&=&
\frac{2 n \left(3 n^2-29 n+62\right)}{n-1} >0\\
\varphi \left( \binom{n}{2} \right)&=&
\frac{n(n-5)^2}{2} > 0.
\end{eqnarray*}
Thus 
$\varphi(x) >0$ for all $6n \le x \le \binom{n}{2}$.
Since $6n \le m \le \binom{n}{2}$, it follows that
$
c_4(G) \le \binom{m-n+1}{2} 
$.

(iii)
Note that $n\ge 7$ since $5 \le n-2$.
Since $5 \le \delta \le \Delta \le n-2$,
by substituting $\alpha=5$ and $\beta = n-2$ in the inequality 
in Proposition \ref{omegadelta} and (\ref{referee1}), 
we have
\begin{eqnarray*}
& & 
\binom{m-n+1}{2}-c_4(G)\\
& \ge &
\binom{m-n+1}{2}-
\left(
\frac{( m- \frac{5n}{2})^2}{2}\left(1 - \frac{1}{n-2}\right)+\frac{m(n-24)}{4}
+\frac{3}{4}
 \left(2m (n+3) -5n (n-2) \right)
\right)\\
&=&
\frac{1}{2(n-2)}
\left( m^2 
-\frac{1}{2} \left(n^2+16\right)m
+
\frac{1}{4} n \left(9 n^2-57 n+128\right)
\right).
\end{eqnarray*}
The minimum value of this function
is
$$
\frac{1}{32} \left(-n^3+34 n^2-192 n+128\right)
=
\frac{1}{32} \left((27-n) \left(n^2-7 n+3\right)+47\right) 
$$ 
when $m = \frac{1}{4} \left(n^2+16\right)$
and this minimum value is positive for $n=7,8,\dots,27$.
\end{proof}

\begin{Proposition}
\label{induction}
Let $G$ be a $2$-connected graph with $n \ge 6$ vertices and $m$ edges that has at least one odd cycle.
Let $ \delta =\delta(G) \ge 2$ and $\Delta = \Delta(G)$.
Suppose that 
$c_4(H) \le \binom{m'-n'+1}{2}$
for every graph $H$ with $n'$ vertices and $m'$ edges
having at least one odd cycle
obtained by deleting either a vertex or an edge of $G$. 
If $G$ satisfies at least one of the following conditions, then $c_4(G) \le \binom{m-n+1}{2}${\rm :}
\begin{enumerate}
\item[{\rm (i)}]
$\Delta = n-1${\rm ;}
\item[{\rm (ii)}]
$\delta \ge 3$ and 
$m <  \delta(n-\delta) ${\rm ;}
\item[{\rm (iii)}]
$
m \ge  \frac{ ( \delta + 2)  n - (3 \delta +1)}{2} +\frac{1}{2(\delta-1)}
${\rm ;}

\item[{\rm (iv)}]
$\delta \le 3$.
\end{enumerate}
\end{Proposition}

\begin{proof}
First, we explain the reason why we may assume that
any subgraph $H$ of $G$ obtained by 
removing either a vertex or an edge from $G$
satisfies $c_4(H) \le \binom{m'-n'+1}{2}$.
Let $H$ be 
an induced subgraph of $G$ obtained by removing 
a vertex $v\in V(G)$ from $G$.
If $H$ has no triangles, then $G$ has no $K_4$,
and hence Proposition~\ref{withoutk4} guarantees $c_4(G) \le \binom{m-n+1}{2}$.
Thus we may assume that such an induced subgraph $H$
has a triangle, and hence $H$ 
satisfies $c_4(H) \le \binom{m'-n'+1}{2}$
by the hypothesis.
On the other hand, 
let $H$ be 
a subgraph of $G$ obtained by removing 
an edge $e_0\in E(G)$ from $G$.
By the same argument in Claim 1 in the proof of Proposition~\ref{withoutk4},
we may assume that $H$ is not bipartite.
By the hypothesis, 
$H$ satisfies 
$c_4(H) \le \binom{m'-n'+1}{2}$.

In addition, by the same argument in Claim 2 in the proof of Proposition~\ref{withoutk4}, 
we may assume that any edge $e_0$ is contained in at least $(m-n+1)$ 4-cycles in $G$.
This fact will play an important role for (ii) and (iv).

(i) 
Let $v$ be a vertex of $G$ of degree $n-1$ and 
let $G'$ an induced subgraph of $G$ obtained by removing $v$ from $G$.
Then $G'$ is connected, not bipartite, and
has $n-1$ vertices and $m-n+1$ edges.
Then $G'$ satisfies 
$c_4(G') \le \binom{(m-n+1) - (n-1) + 1}{2}$, 
and by applying (\ref{referee2}) to 
$\sum_{v' \in V(G')} \deg_{G'} (v')^2$,
it follows that
\begin{eqnarray*}
c_4(G) &=& c_4(G') + \sum_{v' \in V(G')} \binom{\deg_{G'} (v')}{2} \\
&\le &  \binom{m-2n+3}{2}
+ \frac{1}{2}\sum_{v' \in V(G')} \left(\deg_{G'} (v')^2 -\deg_{G'} (v')\right)\\
&\le& \binom{m-2n+3}{2} +
\frac{1}{2} (m-n+1) \left(\frac{2 (m-n+1)}{n-2}+n-3\right)
-(m-n+1)\\
&= & \binom{m-n+1}{2} - \frac{\left(\binom{n}{2}-m\right) (m-2 n+3)}{n-2}.
\end{eqnarray*}
Since $G'$ is a connected graph with $n-1$ vertices and $m-n+1$ edges,
it follows that $(n- 1) -1 \le m-n+1$.
Hence $2n -3 \le m \le \binom{n}{2} $.
Thus $c_4(G) \le \binom{m-n+1}{2}$.

(ii)
Suppose that $G$ satisfies $c_4(G) > \binom{ m-n+1 }{2}$, $\delta \ge 3$ and 
$m < \delta(n-\delta) $.
Let $v$ be a vertex of $G$ such that the degree of a vertex $v$ is $\delta$.
Suppose that $v_1,\ldots, v_\delta$ are incident with $v$ in $G$.
Let $\deg_G(v_i) = \delta + \alpha_i$ for $i=1,2,\dots,\delta$ and $\alpha = \min\{\alpha_1,\ldots, \alpha_\delta\}$.
Then 
\begin{equation}
\frac{1}{2}\delta(n+\alpha) 
\le
\frac{1}{2}
\left(
\sum_{i=1}^\delta (\delta + \alpha) 
+
\sum_{i=\delta + 1}^n \delta 
\right)
\le \frac{1}{2}\sum_{i=1}^n \deg_G(v_i) =  m. \label{referee3}
\end{equation}
Suppose that $\alpha = \alpha_j$.
Let
$c_t (v,v_j)$ denote the number of $t$-cycles of $G$ containing $\{v, v_j\}$.
Then we have
$$m-n+1 \le c_4(v,v_j).$$
In addition, as in the proof of Proposition~\ref{omegadelta},
it follows that
$$c_4 (v,v_j) \le 
(\deg_G(v) - 1)(\deg_G(v_j) - 1)  - c_3(v,v_j)
\le (\deg_G(v) - 1)(\deg_G(v_j) - 1).
$$
Thus
$$
m-n+1 \le  c_4(v,v_j) \le (\deg_G (v) - 1)(\deg_G(v_j) - 1)
=(\delta - 1)(\delta + \alpha - 1)
$$
and hence
\begin{equation}
m \le (\delta - 1)(\delta + \alpha - 1) + n - 1.\label{referee4}
\end{equation}
From (\ref{referee3}) and (\ref{referee4}),
$$
\frac{1}{2} \delta(n+\alpha) \le m \le (\delta - 1)(\delta + \alpha - 1) +n-1.
$$
This inequality simplifies to
$
(\delta -2) (n -2\delta)
\le
(\delta - 2)\alpha
$.
Since $\delta \ge 3$ by hypothesis, 
canceling implies
$
n -2 \delta \le \alpha 
$.
Thus $2 (n-\delta)  \le n+\alpha$ and hence
$$
\delta(n-\delta) \le 
\frac{1}{2}\delta(n+\alpha) \le  m,$$
which contradicts the hypothesis $m < \delta(n-\delta)$.

(iii)
We may assume that $\Delta \le n-2$ by (i).
Let $v$ be a vertex of $G$ of degree $\delta$
that is adjacent to $v_1,\ldots, v_\delta$
and let $G'$ an induced subgraph of $G$ obtained by deleting $v$.
Note that the number of copies of $C_4$ that contains $v$ is equal to $\sum_{1 \le i < j \le \delta} s_{ij}$,
where 
$$s_{ij}= | \{v' \ | \  \{v_i,v'\}, \{v_j,v'\} \in E(G), v' \ne v \}|.$$
It is trivial that $s_{ij} \le n-3$.
On the other hand, if $s_{ij} = s_{ik} = n-3$ ($j \ne k$), then
the degree of $v_i$ is $n-1$.
Hence the number of $i , j $ such that $s_{ij} = n-3$ is at most $\lfloor \delta/2 \rfloor $.
Thus the number of copies of $C_4$ that contains $v$ is at most $(n-4) \binom{\delta}{2} +  \left\lfloor \frac{\delta}{2} \right\rfloor$.
Therefore
$$
c_4(G) \le c_4(G') + (n-4) \binom{\delta}{2} +\left\lfloor \frac{\delta}{2} \right\rfloor.
$$
Since $G'$ satisfies 
$ c_4(G') \le \binom{(m-\delta)-(n-1)+1}{2}$,
we have
\begin{eqnarray*}
 &  & \binom{m-n+1}{2} - c_4(G) \\
& \ge & \binom{m-n+1}{2} - \binom{(m-\delta)-(n-1)+1}{2}- (n-4) \binom{\delta}{2} - \left\lfloor \frac{\delta}{2} \right\rfloor \\
&=& 
(\delta-1)
\left( m - \frac{ ( \delta + 2)  n - (3 \delta +1)}{2}\right) - \frac{1}{2}
+
\frac{\delta}{2}
-
\left\lfloor \frac{\delta}{2} \right\rfloor.
\end{eqnarray*}
This is nonnegative if 
$$
m \ge  
\begin{cases}
\frac{ ( \delta + 2)  n - (3 \delta +1)}{2} +\frac{1}{2(\delta-1)}
& \ \ \  \delta \mbox{ is even,}\\
\frac{ ( \delta + 2)  n - (3 \delta +1)}{2} & \ \ \ \delta \mbox{ is odd.}
\end{cases}
$$

(iv)
Let $G$ be a 2-connected graph with minimum degree $\delta \in \{2,3\}$
having at least one odd cycle
such that
$$c_4(G) > \binom{ m-n+1 }{2}.$$

Suppose that $\delta =2$,
%
$\deg (v_1)=2$, and
$v_1$ is contained in $k$ 4-cycles of $G$.
Then $0 \le k \le n-3$.
Since each edge of $G$ is contained in at least $(m-n+1)$ 4-cycles of $G$, 
we have $m-n+1 \le k$ and hence $m \le n+k-1$.
The union of $k$ $C_4$'s is a complete bipartite graph $K_{2,k+1}$ with $2(k+1)$ edges.
If $k=n-3$, then $G$ has at least $2(k+1)+1=2k+3=n+k$ edges since $G$ has at least one odd cycle.
If $k <n-3$, then $G$ has at least $2(k+1) + (n-k-3) +1= n+k$ edges
 since remaining $n-k-3$ vertices are of degree $\ge 2$
and $G$ is 2-connected.
Thus we have $n+k \le m \le  n+k-1$, a contradiction.

Suppose that $\delta = 3$.
By (ii) and (iii), we may assume that 
\begin{equation}
\label{last}
3(n-3) \le m < \frac{ ( \delta + 2)  n - (3 \delta +1)}{2}
=
\frac{5}{2} n -5.
\end{equation}
Since $3 n-9< \frac{5}{2} n -5$, we have $n <8$ and hence $n = 6,7$.
If $n=6$, then $m = 9$ by (\ref{last}).
Then $G$ is a connected $3$-regular graph with $6$ vertices.
It is easy to see that $G$ has no $K_4$, a contradiction.
If $n=7$, then $m = 12$ by (\ref{last}).
Let $\{v_1, \ldots, v_7\}$ be the vertex set of $G$
with $\deg_G(v_1)=3$ and
$\{v_1,v_2\},\{v_1,v_3\},\{v_1,v_4\} \in E(G)$.
Since $m = \delta(n-\delta)$, by the argument in the proof of (ii),
the degree of $v_i$ ($i=2,3,4$) is $\delta+\alpha= \delta + n - 2\delta = 4$.
Then 
$$24= 2m =\sum_{i=1}^7 \deg_G (v_i)=  15+ \deg_G (v_5) + \deg_G (v_6) + \deg_G (v_7)
.$$
Since $\deg_G (v_i) \ge 3$, we have $\deg_G (v_i)=3$ for $i = 5,6,7$.
Note that the degree of a vertex incident 
with one of the vertices $v_1, v_5, v_6, v_7$ of degree 3 
should be 4.
Thus $G$ is the complete bipartite graph
$K_{3,4}$ on the vertex set $\{v_2,v_3,v_4\} \cup\{v_1, v_5, v_6, v_7\} $,
a contradiction.
\end{proof}

We are now in a position to prove the main theorem of the present paper.

\begin{proof}[Proof of Theorem~\ref{main}]
Let $G$ be a 2-connected graph with $n \ge 6$ vertices having at least one odd cycle
with
\[
c_4(G) >\binom{m-n+1}{2}.
\]
Suppose that the assertion holds for any connected graph $G'$
having at least one odd cycle
obtained by deleting either a vertex or an edge of $G$. 
For $n \in \{6,7,8, 9\}$, we have $\left\lfloor (3n+1)/4 \right\rfloor=n-2$.
By Proposition~\ref{deltafour} (i) and Proposition~\ref{induction} (i) and (iv), 
it follows that $n\ge 10$.
In addition, by 
Proposition~\ref{deltafour} (ii) and Proposition~\ref{induction},
we may assume that $G$ is a 2-connected graph with 
$4 \le \delta \le 11$, $\Delta \le n-2$, and 
$$
\delta(n-\delta) \le m <\frac{ ( \delta + 2)  n - (3 \delta +1)}{2} +\frac{1}{2(\delta-1)}
.$$
In particular, 
$$
\delta(n-\delta) <
\frac{ ( \delta + 2)  n - (3 \delta +1)}{2} +\frac{1}{2(\delta-1)}
$$
holds.
Since $\delta \ge 4$, we have
$$
n <
\frac{2 \delta^3-5 \delta^2+2 \delta+2}{(\delta-2) (\delta-1)}
=2\delta +1 +\frac{\delta}{(\delta-2) (\delta-1)}
<2\delta +2.
$$
If $\delta = 4$, then 
$
n < 2\delta +2 =10,
$
a contradiction.
Thus $\delta \ge 5$.
By Proposition \ref{deltafour} (iii), we have $n \ge 28$.
However, since $\delta \le 11$, we have
$n < 2\delta +2 \le 24.
$
This is a contradiction.
\end{proof}

It would, of course, be of interest to classify all connected graphs $G$ which satisfy the equality $c_4(G) = \binom{m - n + 1}{2}$.

\appendix

\section{Algebraic proof for equation (\ref{Sanda})}
\label{omake}

In this appendix, we give an algebraic proof of equation (\ref{Sanda}) in Section 1.
Let $G$ be a connected graph on the vertex set $\{v_1,\ldots, v_n\}$ whose 
edge set is $\{e_1,\ldots,e_m\}$.
Then the toric ideal $I_G$ of the edge ring $K[G]$ of $G$ is defined as follows.
Let $K[x_1,\ldots,x_m]$ and $K[t_1,\ldots,t_n]$ be
polynomial rings over a field $K$.
Define the ring homomorphism $\pi_G :  K[x_1,\ldots,x_m] \rightarrow K[t_1,\ldots,t_n]$
by $\pi_G (x_i) = t_j t_k \in K[t_1,\ldots,t_n]$ where $e_i = \{v_j, v_k\}$ for each $1 \le i \le m$,.
Then the {\em toric ring} $K[G]$ is the image of $\pi_G$,
and the {\em toric ideal} $I_G$ of $K[G]$ is the kernel of $\pi_G$.
See  \cite[Chapter 5]{binomialideals} for details.
It is known 
 \cite[Lemma 5.9]{binomialideals}
that the toric ideal $I_G$ is generated by homogeneous binomials of the form
$$
f_\Gamma = \prod_{k=1}^\ell x_{i_{2k-1}} - \prod_{k=1}^\ell x_{i_{2k}},
$$
where $\Gamma=(e_{i_1},\ldots,e_{i_{2\ell}})$ is a closed walk of even length in $G$.
In particular, $f_\Gamma \in I_G$ is quadratic if and only if $\Gamma$ is a 4-cycle contained in $G$.

Applying a result by Herzog et al.
\cite[Corollary~2.6]{HSZ} to the edge ring of $G$,
it follows that the number of quadratic binomials in a minimal set of generators of $I_G$
is less than or equal to $\binom{m-\dim K[G] +1}{2}$, where
$$
\dim K[G] =
\begin{cases}
n-1 & \mbox{if } G \mbox{ is bipartite,}\\
n &  \mbox{otherwise.}
\end{cases}
$$
If $f_\Gamma \in I_G$ is quadratic, then $f_\Gamma$ is generated by other (quadratic) binomials of $I_G$
if and only if the induced subgraph $G'$ of $G$ on the vertex set $V(\Gamma)$ is a complete graph $K_4$.
More precisely, if $G'$ is $K_4$, then $G'$ has three 4-cycles $\Gamma, \Gamma_1, \Gamma_2$
and $f_\Gamma = f_{\Gamma_1} + f_{\Gamma_2}$.
Thus the number of quadratic binomials in a minimal set of generators of $I_G$
is $c_4(G)-k_4(G)$.
If $G$ is bipartite, then $k_4(G)=0$.
Hence
$$
c_4(G)
 \leq 
\begin{cases}
\binom{m-n+2}{2} & \mbox{if } G  \mbox{ is bipartite},\\
\\
\binom{m-n+1}{2} + k_4(G) & \mbox{otherwise.}
\end{cases}
$$

\section*{Acknowledgments}
The authors would like to thank the referees for 
their careful reading and 
important advise that 
improve the writing of the present paper.

\end{document}